\documentclass[pdflatex,sn-mathphys-num]{sn-jnl}

\usepackage{graphicx}
\usepackage{amsmath,amssymb,amsfonts}
\usepackage{amsthm}
\usepackage{mathrsfs}
\usepackage{enumitem}
\usepackage[title]{appendix}
\usepackage{xcolor}
\usepackage{manyfoot}
\usepackage{booktabs}
\usepackage{listings}

\theoremstyle{thmstyleone}
\newtheorem{theorem}{Theorem}
\newtheorem{lemma}[theorem]{Lemma}

\newtheorem{proposition}[theorem]{Proposition} 

\theoremstyle{thmstyletwo}

\newtheorem{remark}{Remark}

\theoremstyle{thmstylethree}

\DeclareMathOperator*{\argmin}{arg\,min}

\DeclareMathOperator{\prox}{prox}

\newcommand{\qbox}[1]{\quad\hbox{#1}\quad}

\newcommand{\R}{\mathbb{R}}
\newcommand{\wbx}{\hfill$\square$}

\newcommand{\ackname}{Acknowledgements}
\def\acknowledgement{\par\addvspace{17pt}\small\rmfamily
\trivlist\if!\ackname!\item[]\else
\item[\hskip\labelsep
{\bfseries\ackname}]\fi}

\newenvironment{acknowledgements}{\begin{acknowledgement}}
{\end{acknowledgement}}

\begin{document}

\title{ Accelerated Backward Forward Method for Convex Optimization }

\author{\fnm{Zepeng} \sur{Wang}}\email{zepeng.wang@rug.nl}

\author{\fnm{Juan} \sur{Peypouquet}}\email{j.g.peypouquet@rug.nl}

\affil{\orgdiv{Bernoulli Institute for Mathematics, Computer Science and Artificial Intelligence}, \orgname{University of Groningen}, \state{Groningen}, \country{The Netherlands}}

\abstract{ We analyze the convergence rate of an accelerated backward forward method for solving convex composite optimization problems. The method was developed by Taylor, Hendrickx and Glineur, and is different from the FISTA algorithm in its placement of the proximal operator. When the smooth part of the objective function is convex, we establish a fast convergence rate of $\mathcal{O}\left( \frac{1}{k^2} \right)$ for the function values, and prove the weak convergence of the iterates. When the smooth part is strongly convex, we propose a variant of the method, and establish an accelerated linear convergence rate for the function values. }


\maketitle

\section{Introduction}
Let $H$ be a real Hilbert space, and consider the convex composite optimization problem:
\begin{equation}\label{Problem: cvx_composite}
\min_{x\in H} F(x) = f(x) + g(x),
\end{equation}
where $f:H\to\mathbb{R}$ is convex and $L$-smooth; $g:H\to\mathbb{R}\cup\{\infty\}$ is convex, proper and lower-semicontinuous. We assume that the solution set $\mathcal{S}=\argmin_{x\in H} F(x)$ is nonempty and denote $F^* = \min_{x\in H}F(x)$. Throughout this paper, the proximal operator $\prox_g:H\to H$, associated with the function $g$, is defined by
$$ \prox_g(z) := \argmin_{x\in H}\left\{ g(x) + \frac{1}{2}\| x - z \|^2 \right\},\qbox{for} z\in H. $$

To solve problem \eqref{Problem: cvx_composite}, one can employ the accelerated forward backward method, also known as FISTA \cite{Beck_2009}:
\begin{equation}\label{Algo: FISTA}\tag{FISTA}
\left\{
\begin{aligned}
y_{k+1} &= \prox_{sg}\left( x_k - s\nabla f(x_k) \right),\\
x_{k+1} &= y_{k+1} + \frac{k}{k+\alpha}(y_{k+1}-y_k),
\end{aligned}
\right.
\end{equation}   
where $k\ge 0$, $\alpha\ge 3$ and $s\in\left( 0,\frac{1}{L} \right]$. It was shown that $F(y_{k+1}) - F^* \le \mathcal{O}\left( \frac{1}{k^2} \right)$ when $\alpha\ge 3$ \cite{Beck_2009}, and that $F(y_{k+1}) - F^* \le o\left( \frac{1}{k^2} \right)$ when $\alpha>3$ \cite{Attouch_2016}. A weak convergence of the iterates generated by \eqref{Algo: FISTA} can be guaranteed for all $\alpha\ge 3$ \cite{Chambolle_2015,Attouch_2018,Ryu_2025,Radu_2025_iterate,Bauschke_2026}. If $f$ is $\mu$-strongly convex, \eqref{Algo: FISTA} can be modified as \cite{Beck_2017}
\begin{equation}\label{Algo: FISTA-SC}\tag{FISTA-SC}
\left\{
\begin{aligned}
y_{k+1} &= \prox_{sg}\left( x_k - s\nabla f(x_k) \right),\\
x_{k+1} &= y_{k+1} + \frac{1-\sqrt{\mu/L}}{1+\sqrt{\mu/L}}(y_{k+1}-y_k),
\end{aligned}
\right.
\end{equation}
and provides an accelerated linear convergence rate of $\mathcal{O}\left( (1-\sqrt{\mu/L})^k \right)$ for the function values \cite{Beck_2017}. 

The algorithms \eqref{Algo: FISTA} and \eqref{Algo: FISTA-SC} generalize Nesterov's accelerated gradient method \cite{Nesterov_1983,Nesterov_2004} to the forward backward setting. Compared to forward backward algorithms, backward forward algorithms have received relatively less attention \cite{Attouch_2018_BF}. An alternative method to solve \eqref{Problem: cvx_composite} is the accelerated {\it backward forward} algorithms, one of which is \cite[Algorithm FPGM2]{Taylor_2017}:
\begin{equation}\label{Algo: APM}\tag{ABF}
\left\{
\begin{array}{rcl}
y_{k+1} &=& x_k - s \nabla f(x_k), \\[3pt]
z_{k+1} &=& y_{k+1} + \lambda_{k+1} ( y_{k+1} - y_k ) + \frac{\lambda_{k+1}s}{\gamma_k}( z_k - x_k ),\\[3pt]
x_{k+1} &=& \prox_{ \gamma_{k+1} g }(z_{k+1}),
\end{array} 
\right.
\end{equation}
where $k\ge 0$, $s\in\left( 0,\frac{1}{L} \right]$, $x_0=\prox_{\gamma_0 g}(z_0)$, $\gamma_{k+1}=(1 + \lambda_{k+1})s$ and $(\lambda_{k+1})_{k \ge 0}$ is a positive sequence. 

The performance estimation (see \cite{Drori_2014}) analysis carried out in \cite{Taylor_2017} shows a slightly---but strictly!---better worst-case convergence rate for \eqref{Algo: APM} with respect to \eqref{Algo: FISTA}. However, no theoretical convergence proof has been established, and the strongly convex case has not been studied.

In this paper, we revisit \eqref{Algo: APM} and carry out the corresponding convergence analysis. In the convex case, we establish a theoretical convergence rate guarantee for the function values, that exactly matches the one known for \eqref{Algo: FISTA}, showing that \eqref{Algo: APM} is at least as fast, and confirming the numerical findings in \cite{Taylor_2017}. We also prove the weak convergence of the iterates to a minimizer of the objective function. In the case where $f$ is strongly convex, we propose an adaptation of \eqref{Algo: APM} that achieves the accelerated linear convergence rate for the function values which is typical of \eqref{Algo: FISTA-SC}. 

The paper is structured as follows: In Section \ref{Sec: fixed_point}, we briefly discuss the fixed-point properties of \eqref{Algo: APM}, especially in connection with the solutions of \eqref{Problem: cvx_composite}. Section \ref{Sec: estimations} is technical, and contains estimations for the function value gap and its decrease along the iterations. Under convexity, the accelerated convergence rate result is given in Section \ref{Sec: convexity}. The weak convergence of the sequences generated by \eqref{Algo: APM} is shown in Section \ref{Sec: iterates}. Section \ref{Sec: strong convexity} is devoted to the study of a linearly convergent backward-forward method for strongly convex problems. We conclude in Section \ref{Sec: conclusions}.

\section{Fixed-point and asymptotic properties} \label{Sec: fixed_point}
In \eqref{Algo: FISTA}, the (constant) step sizes used in the proximal and the gradient descent operators are the same, namely $s$. As a consequence, the fixed points of \eqref{Algo: FISTA} are precisely the solutions of \eqref{Problem: cvx_composite}. On the other hand, \eqref{Algo: APM} uses two different step sizes, namely $s$ and $\gamma_{k+1}=(1+\lambda_{k+1})s$, one of which can be iteration-dependent. In the convex case, our choice of parameters implies that $\lambda_{k+1}\to 1$, whence $\gamma_{k+1}\to 2s$ (not $s$) as $k\to\infty$. 

Suppose $(x^*,y^*,z^*)$ is an {\it asymptotic} fixed point of \eqref{Algo: APM}, in the sense that the iteration-dependent parameters are replaced by their limits. This gives:
$$\left\{
\begin{array}{rcl}
y^* &=& x^* - s \nabla f(x^*) \\
z^* &=& y^* + ( y^* - y^* ) + \frac{1}{2}( z^* - x^* ) \\
x^* &=& \prox_{ 2s g }(z^*),
\end{array} 
\right.
$$
which is equivalent to
$$\left\{
\begin{array}{rcl}
2x^*-2y^* &=&  2s \nabla f(x^*) \\
2y^*-x^*-z^* &=& 0 \\
z^*-x^* &\in& 2s\partial g(x^*).
\end{array} 
\right.
$$
This easily implies that
$$0\in\nabla f(x^*)+\partial g(x^*),$$
and so $x^*$ is a minimizer of $F$. It also follows that $y^*=x^*-s\nabla f(x^*)$ and $z^*=x^*-2s\nabla f(x^*)$, respectively. These are the relationships between the limits of the sequences generated by \eqref{Algo: APM}, provided they exist (which they do, as we shall see later).

\begin{remark} \label{Remark:nablaf_constant}
The set $\mathcal S$ of solutions to problem \eqref{Problem: cvx_composite} may have multiple points. However, the gradient of $f$ is constant on $\mathcal S$, in view of the present hypotheses. Indeed, let $x_1^*,x_2^*\in\mathcal S$, so that $-\nabla f(x_i^*)\in\partial g(x_i^*)$ for $i=1,2$. Since $f$ is $L$-smooth and convex, and $g$ is convex, we have
\begin{align*}
    f(x_1^*) & \ge f(x_2^*)+\langle\nabla f(x_2^*),x_1^*-x_2^*\rangle+\frac{1}{2L}\|\nabla f(x_1^*) -\nabla f(x_2^*)\|^2 \\
    g(x_1^*) & \ge g(x_2^*)-\langle\nabla f(x_2^*),x_1^*-x_2^*\rangle. 
\end{align*}
Adding the two inequalities, we deduce that 
$$F(x_1^*)  \ge F(x_2^*)+\frac{1}{2L}\|\nabla f(x_1^*) -\nabla f(x_2^*)\|^2.$$
Interchanging the roles of $x_1^*$ and $x_2^*$, we deduce that $\nabla f(x_1^*) =\nabla f(x_2^*)$.
\end{remark}

\section{Estimations for the function value gap and its decrease}\label{Sec: estimations}
\begin{proposition}\label{Prop: F_estimate}
Let $f:H\to\R$ be $\mu$-strongly convex and $L$-smooth. Consider algorithm \eqref{Algo: APM} with $s\in\left( 0,\frac{1}{L} \right]$, and set $x^*\in\mathcal{S}$. Then, for every $k\ge 0$, we have
\begin{align*}
F(x_{k+1}) - F(x_k) 
&\le - \frac{1}{2s(1-\mu s)} \| y_{k+2} - y_{k+1} \|^2 
    - \frac{1}{2s}\| x_{k+1} - x_k \|^2 \\
&\quad + \frac{\lambda_{k+1}}{s} \langle y_{k+1}-y_k, x_{k+1} - x_k \rangle, 
\end{align*}
and
\begin{align*}
F(x_{k+1}) - F^* + (1+\lambda_{k+1})\eta_{k+1} - \lambda_{k+1} \eta_k + \psi_{k+1} 
&\le - \frac{1}{s} \langle y_{k+2}-y_{k+1}, x_{k+1}-x^* \rangle \\
&\quad + \frac{\lambda_{k+1}}{s} \langle y_{k+1}-y_k, x_{k+1} - x^* \rangle,
\end{align*}
where
\begin{align*}
\eta_k &:= \left\langle \frac{z_k-x_k}{\gamma_k}, x_k - x^* \right\rangle - \left( g(x_k) - g(x^*) \right)\ge 0,\\
\psi_k &:= \langle \nabla f(x_k), x_k - x^* \rangle - \left( f(x_k) - f(x^*) \right) \ge 0.
\end{align*}
\end{proposition}

\begin{proof}
From the update rule for $x_{k+1}$ and $z_{k+1}$ in \eqref{Algo: APM}, we obtain
$$ \frac{z_{k+1} - x_{k+1}}{\gamma_{k+1}} \in \partial g(x_{k+1}), $$
and
\begin{equation}\label{E: subgradient_iterate}
\frac{z_{k+1}-x_{k+1}}{\gamma_{k+1}} 
= \frac{ -(x_{k+1}-y_{k+1}) + \lambda_{k+1} (y_{k+1}-y_k) }{(1 + \lambda_{k+1})s} + \left( \frac{\lambda_{k+1}}{1 + \lambda_{k+1}} \right) \frac{z_k-x_k}{\gamma_k}.
\end{equation}
Multiplying both sides by $1+\lambda_{k+1}$ gives
\begin{equation}\label{E: iterates}
\frac{z_{k+1} - x_{k+1}}{\gamma_{k+1}} + \lambda_{k+1} \left( \frac{z_{k+1} - x_{k+1}}{\gamma_{k+1}} - \frac{z_k-x_k}{\gamma_k} \right)  
= \frac{-(x_{k+1}-y_{k+1}) + \lambda_{k+1} (y_{k+1}-y_k)}{s}.
\end{equation}
Taking inner product with $x_{k+1}-x_k$ and using the convexity of $g$, we deduce that
$$ g(x_{k+1}) - g(x_k) \le - \frac{1}{s} \langle x_{k+1}-y_{k+1}, x_{k+1} - x_k \rangle + \frac{\lambda_{k+1}}{s} \langle y_{k+1}-y_k, x_{k+1} - x_k \rangle.  $$
Since $f$ is $\mu$-strongly convex and $L$-smooth, we know from \cite[Theorem 4]{Talor_2017_SC} that
\begin{align*}
&\quad f(x_{k+1}) - f(x_k) \\ 
&\le \langle \nabla f(x_{k+1}), x_{k+1} - x_k \rangle - \frac{s}{2}\| \nabla f(x_{k+1}) - \nabla f(x_k) \|^2 - \frac{\mu}{2(1-\mu s)}\| y_{k+2} - y_{k+1} \|^2\\
&= -\frac{1}{s} \langle y_{k+2}-x_{k+1}, x_{k+1} - x_k \rangle - \frac{s}{2}\| \nabla f(x_{k+1}) - \nabla f(x_k) \|^2 - \frac{\mu}{2(1-\mu s)}\| y_{k+2} - y_{k+1} \|^2.
\end{align*}
Summing up the last two inequalities gives
\begin{align*}
F(x_{k+1}) - F(x_k)   
&\le -\frac{1}{s} \langle y_{k+2}-y_{k+1}, x_{k+1} - x_k \rangle + \frac{\lambda_{k+1}}{s} \langle y_{k+1}-y_k, x_{k+1} - x_k \rangle \\
&\quad - \frac{s}{2}\| \nabla f(x_{k+1}) - \nabla f(x_k) \|^2
       - \frac{\mu}{2(1-\mu s)}\| y_{k+2} - y_{k+1} \|^2. 
\end{align*}
Using
$$ \nabla f(x_{k+1}) = -\frac{1}{s}(y_{k+2}-x_{k+1}),\quad \nabla f(x_k) = -\frac{1}{s}(y_{k+1}-x_k), $$
we obtain
\begin{align*}
&\quad \frac{s}{2}\| \nabla f(x_{k+1}) - \nabla f(x_k) \|^2 \\
&= \frac{1}{2s} \| (y_{k+2} - y_{k+1}) - (x_{k+1} - x_k) \|^2 \\
&= \frac{1}{2s} \| y_{k+2} - y_{k+1} \|^2 + \frac{1}{2s}\| x_{k+1} - x_k \|^2 
   - \frac{1}{s}\langle y_{k+2} - y_{k+1}, x_{k+1} - x_k \rangle,
\end{align*}
so that
\begin{align*}
- \frac{1}{s}\langle y_{k+2} - y_{k+1}, x_{k+1} - x_k \rangle 
&= \frac{s}{2}\| \nabla f(x_{k+1}) - \nabla f(x_k) \|^2
   - \frac{1}{2s} \| y_{k+2} - y_{k+1} \|^2 \\ 
&\quad - \frac{1}{2s}\| x_{k+1} - x_k \|^2.
\end{align*}
This results in
\begin{align*}
F(x_{k+1}) - F(x_k) 
&\le - \frac{1}{2s(1-\mu s)} \| y_{k+2} - y_{k+1} \|^2 
    - \frac{1}{2s}\| x_{k+1} - x_k \|^2 \\
&\quad + \frac{\lambda_{k+1}}{s} \langle y_{k+1}-y_k, x_{k+1} - x_k \rangle.
\end{align*}

Likewise, we take inner product of \eqref{E: iterates} with $x_{k+1}-x^*$ to obtain
\begin{align*}
&\quad (1+\lambda_{k+1}) \left\langle \frac{z_{k+1} - x_{k+1}}{\gamma_{k+1}}, x_{k+1}-x^*\right\rangle
  -\lambda_{k+1} \left\langle \frac{z_k-x_k}{\gamma_k}, x_{k+1}-x^*\right\rangle \\
&= - \frac{1}{s} \langle x_{k+1}-y_{k+1}, x_{k+1}-x^* \rangle + \frac{\lambda_{k+1}}{s} \langle y_{k+1}-y_k, x_{k+1} - x^* \rangle.
\end{align*}
Since $ \frac{z_k-x_k}{\gamma_k} \in \partial g(x_k)$ and $g$ is convex, we have
\begin{align*}
\left\langle \frac{z_k-x_k}{\gamma_k}, x_{k+1}-x^*\right\rangle
&= \left\langle \frac{z_k-x_k}{\gamma_k}, x_k-x^*\right\rangle
  + \left\langle \frac{z_k-x_k}{\gamma_k}, x_{k+1}-x_k\right\rangle\\
&\le \left\langle \frac{z_k-x_k}{\gamma_k}, x_k - x^* \right\rangle  + g(x_{k+1}) - g(x_k),
\end{align*}
which results in
\begin{align*}
&\quad \left(g(x_{k+1}) - g(x^*)\right)
  + (1+\lambda_{k+1})\left[ \left\langle \frac{z_{k+1}-x_{k+1}}{\gamma_{k+1}}, x_{k+1} - x^* \right\rangle - \left( g(x_{k+1}) - g(x^*) \right) \right]\\
&\quad -\lambda_{k+1} \left[ \left\langle \frac{z_k-x_k}{\gamma_k}, x_k - x^* \right\rangle - \left( g(x_k) - g(x^*) \right) \right] \\ 
&\le - \frac{1}{s} \langle x_{k+1}-y_{k+1}, x_{k+1}-x^* \rangle + \frac{\lambda_{k+1}}{s} \langle y_{k+1}-y_k, x_{k+1} - x^* \rangle.
\end{align*}
Combining this inequality with
\begin{align*}
&\quad f(x_{k+1}) - f(x^*) \\ 
&= \langle \nabla f(x_{k+1}), x_{k+1} - x^* \rangle - \left[ \langle \nabla f(x_{k+1}), x_{k+1} - x^* \rangle - \left( f(x_{k+1}) - f(x^*) \right) \right] \\
&= -\frac{1}{s}\langle y_{k+2}-x_{k+1}, x_{k+1} - x^* \rangle - \left[ \langle \nabla f(x_{k+1}), x_{k+1} - x^* \rangle - \left( f(x_{k+1}) - f(x^*) \right) \right],   
\end{align*}
we obtain the desired result.
\end{proof}

\section{Convergence rate for convex problems}\label{Sec: convexity}
In this section, we establish a fast convergence rate $\mathcal{O}\left( \frac{1}{k^2} \right)$ for \eqref{Algo: APM} when $f$ is convex, which matches the one obtained for \eqref{Algo: FISTA}. According to the PEP methodology used in \cite{Taylor_2017}, algorithm \eqref{Algo: APM} has a strictly better worst-case convergence rate than \eqref{Algo: FISTA} in practice.

Define the sequence $(t_k)_{k\ge 0}$ by 
\begin{equation}\label{E: t_k}
t_k = \left\{
\begin{array}{ccl}
1,&& \text{if}\quad k=0,\\[3pt]
\frac{m+\sqrt{m^2 + 4t_{k-1}^2}}{2},&& \text{if}\quad k\ge 1,
\end{array}    
\right.
\end{equation}
with $m\in(0,1]$, which satisfies $t_k \ge 1$ and
$$ t_{k+1}(t_{k+1}-m) = t_k^2, $$
and set
$$ \lambda_{k+1} = \frac{t_k-1}{t_{k+1}}. $$

\begin{remark}
If $m=1$, the $(t_k)$ sequence reduces to the one defined by Nesterov's method \cite{Nesterov_1983}. In fact, a general $m\in(0,1]$ gives
$$ t_k^2 \ge t_{k+1}(t_{k+1}-1), $$
which suffices to guarantee the convergence of algorithm \eqref{Algo: APM}. In view of this fact, setting $t_k = \frac{k+\alpha-1}{\alpha-1}$ with $\alpha\ge 3$, we obtain 
$$ \lambda_{k+1} = \frac{k}{k+\alpha},\qbox{with}\alpha\ge 3, $$
using which a similar convergence rate for the function values can be derived. 
\end{remark}

The main convergence result of this section is:

\begin{theorem}\label{Thm: rate_cvx}
Let $F=f+g$, where $f:H\to\mathbb{R}$ is convex and $L$-smooth, and $g:H\to\mathbb{R}\cup\{\infty\}$ is convex, proper and lower-semicontinuous. Let $(x_k)_{k\ge 0}$, $(y_k)_{k\ge 0}$ and $(z_k)_{k\ge 0}$ be generated by algorithm \eqref{Algo: APM}, where 
$$ s\in\left( 0,\frac{1}{L} \right],\quad \lambda_{k+1} = \frac{t_k-1}{t_{k+1}}, $$
with $(t_k)_{k\ge 0}$ defined by \eqref{E: t_k}, and initial conditions 
$$ \gamma_0 = s,\quad y_0\in H,\quad z_0=y_0-\gamma_0 \nabla f(y_0),\quad x_0=\prox_{\gamma_0 g}(z_0). $$
For every $k\ge 0$, we have
$$ F(x_k) - F^* \le \frac{\| y_0-x^* \|^2}{2st_k^2}. $$
Besides, every weak subsequential limit point of $x_k$, as $k\to\infty$, minimizes $F$.
\end{theorem}

Our proof relies on the following energy sequence $(E_k)_{k\ge 0}$:
\begin{equation}\label{E: E_k}
E_k = E_k(x^*) 
    := st_k^2 \left( F(x_k) - F^* \right)
      + s(t_k-1) \eta_k
      + s(t_k-1) \psi_k
      + R_k,
\end{equation}
where
$$ R_k = \frac{t_k(t_k-1)}{2} \| y_{k+1} - y_k \|^2
        + (t_k-1)\langle y_{k+1}-y_k, x_k-x^* \rangle 
        + \frac{1}{2} \| x_k-x^* \|^2.$$

\begin{remark}\label{Rem: R_k_bound}
Observe that $R_k\ge 0$, since $t_k\ge 1$ and
$$ R_k = \frac{1}{2}\left\| (t_k-1)(y_{k+1}-y_k) + (x_k-x^*) \right\|^2 + \frac{t_k-1}{2}\| y_{k+1}-y_k \|^2. $$
\end{remark}

The following result shall be useful in the forthcoming convergence proof.

\begin{lemma}[Lemma 2.3 in \cite{Beck_2009}]\label{Lem: prox_bound}
Let $F:=f+g$, where $f:H\to\R$ is convex and $L$-smooth; $g:H\to\mathbb{R}\cup\{\infty\}$ is convex, proper and lower-semicontinuous. Then, for every $s\in\left( 0, \frac{1}{L} \right]$ and $x,y\in H$,
$$ F\left( y-s G_s(y) \right) \le F(x) + \langle G_s(y), y-x \rangle - \frac{s}{2}\| G_s(y) \|^2, $$
where $G_s(y) = \frac{ y-\prox_{s g}\left( y-s\nabla f(y) \right) }{s}$.
\end{lemma}

\subsection{Proof of Theorem \ref{Thm: rate_cvx}}
Using \eqref{E: E_k}, we obtain
\begin{equation}\label{E: E_k_diff}
\begin{aligned}
E_{k+1} - E_k 
&= st_{k+1}^2 \left( F(x_{k+1}) - F^* \right)
  -st_k^2 \left( F(x_k) - F^* \right) 
  + R_{k+1}-R_k\\
&\quad + s[ (t_{k+1}-1)\eta_{k+1} - (t_k-1)\eta_k ]
  + s[ (t_{k+1}-1)\psi_{k+1} - (t_k-1)\psi_k].  
\end{aligned}
\end{equation} 
Setting $\lambda_{k+1} = \frac{t_k-1}{t_{k+1}}$ and $\mu=0$ in Proposition \ref{Prop: F_estimate}, we obtain
\begin{equation}\label{E: func_gap_C}
\begin{aligned}
&\quad st_{k+1}(t_{k+1}-1) \left( F(x_{k+1}) - F(x_k) \right) + \frac{t_{k+1}(t_{k+1}-1)}{2}\| y_{k+2} - y_{k+1} \|^2 \\
&\quad - \frac{t_k(t_k-1)}{2}\| y_{k+1} - y_k \|^2 \\ 
&\le - \frac{t_k(t_k-1)}{2}\| y_{k+1} - y_k \|^2
     - \frac{t_{k+1}(t_{k+1}-1)}{2}\| x_{k+1} - x_k \|^2 \\
&\quad  + (t_k-1)(t_{k+1}-1) \langle y_{k+1}-y_k, x_{k+1} - x_k \rangle, 
 \end{aligned}
\end{equation}  
and
\begin{equation}\label{E: func_value_C}
\begin{aligned}
& st_{k+1}\left( F(x_{k+1}) - F^* \right) + st_k\eta_{k+1} 
  + st_{k+1}\psi_{k+1} 
  + s[ (t_{k+1}-1)\eta_{k+1} - (t_k-1)\eta_k ] \\
&\le - t_{k+1}\langle y_{k+2}-y_{k+1}, x_{k+1}-x^* \rangle 
     + (t_k-1)\langle y_{k+1}-y_k, x_{k+1}-x^* \rangle \\
&= - (t_{k+1}-1)\langle y_{k+2}-y_{k+1}, x_{k+1}-x^* \rangle 
   + (t_k-1)\langle y_{k+1}-y_k, x_k-x^* \rangle \\
&\quad - \langle y_{k+2}-y_{k+1}, x_{k+1}-x^* \rangle
     + (t_k-1)\langle y_{k+1}-y_k, x_{k+1}-x_k \rangle.    
\end{aligned}
\end{equation}
For the second to last term in \eqref{E: func_value_C}, we have
\begin{align*}
&\quad \langle y_{k+2}-y_{k+1}, x_{k+1}-x^* \rangle \\
&= \langle x_{k+1}-x_k, x_{k+1}-x^* \rangle
  - s \langle \nabla f(x_{k+1})-\nabla f(x_k), x_{k+1}-x^* \rangle \\
&= \frac{1}{2}\| x_{k+1}-x_k \|^2 + \frac{1}{2}\| x_{k+1}-x^* \|^2 
   - \frac{1}{2}\| x_k-x^* \|^2 \\
&\quad -s\langle \nabla f(x_{k+1}), x_{k+1}-x^* \rangle
   + s \langle \nabla f(x_k), x_k-x^* \rangle
   + s \langle \nabla f(x_k), x_{k+1}-x_k \rangle \\
&\ge \frac{1}{2}\| x_{k+1}-x^* \|^2 - \frac{1}{2}\| x_k-x^* \|^2 
   -s\langle \nabla f(x_{k+1}), x_{k+1}-x^* \rangle \\
&\quad + s \langle \nabla f(x_k), x_k-x^* \rangle 
   + s( f(x_{k+1}) - f(x_k) ),
\end{align*}
since $f$ is $L$-smooth and $s\le \frac{1}{L}$ such that
$$ f(x_{k+1}) - f(x_k) \le \langle \nabla f(x_k), x_{k+1}-x_k \rangle + \frac{1}{2s}\| x_{k+1}-x_k \|^2. $$
This gives
\begin{equation}\label{E: key_inequality}
\langle y_{k+2}-y_{k+1}, x_{k+1}-x^* \rangle \ge \frac{1}{2}\| x_{k+1}-x^* \|^2 - \frac{1}{2}\| x_k-x^* \|^2 - s\psi_{k+1} + s \psi_k.
\end{equation}
Using \eqref{E: key_inequality} in \eqref{E: func_value_C}, we obtain
\begin{align*}
&\quad st_{k+1}\left( F(x_{k+1}) - F^* \right) 
  + s(t_{k+1}-1)\psi_{k+1} - s(t_k-1)\psi_k \\ 
&\quad + s[ (t_{k+1}-1)\eta_{k+1} - (t_k-1)\eta_k ] 
  + \left( \frac{1}{2}\| x_{k+1}-x^* \|^2 - \frac{1}{2}\| x_k-x^* \|^2 \right) \\
&\le - (t_{k+1}-1)\langle y_{k+2}-y_{k+1}, x_{k+1}-x^* \rangle 
   + (t_k-1)\langle y_{k+1}-y_k, x_k-x^* \rangle \\
&\quad + (t_k-1)\langle y_{k+1}-y_k, x_{k+1}-x_k \rangle
   - st_k\eta_{k+1} 
   - st_k\psi_k. 
\end{align*}
Using this inequality and \eqref{E: func_gap_C} in \eqref{E: E_k_diff} gives
\begin{equation}\label{E: E_k_diff_bound_temp}
\begin{aligned}
E_{k+1}-E_k  
&\le s[t_{k+1}(t_{k+1}-1)-t_k^2]\left( F(x_k) - F^* \right)
     - st_k\eta_{k+1} 
     - st_k\psi_k \\
&\quad - \frac{t_k(t_k-1)}{2}\| y_{k+1} - y_k \|^2 
   - \frac{t_{k+1}(t_{k+1}-1)}{2}\| x_{k+1} - x_k \|^2 \\ 
&\quad + (t_k-1)t_{k+1} \langle y_{k+1}-y_k, x_{k+1} - x_k \rangle\\
&\le - st_k\eta_{k+1} - st_k\psi_k 
   - \frac{1}{2}\left( 1 - \frac{1}{t_k} \right) \left\| t_k(y_{k+1}-y_k) - t_{k+1}(x_{k+1}-x_k) \right\|^2 \\ 
&\le 0,
\end{aligned}
\end{equation}
in view of $t_k^2 \ge t_{k+1}(t_{k+1}-1)$. It follows that
$$ F(x_k) - F^* \le \frac{E_k}{st_k^2} \le \frac{E_0}{st_k^2}, $$
with
$$ E_0 = s\left( F(x_0) - F^* \right) + \frac{1}{2}\| x_0 - x^* \|^2. $$
Setting $x=x^*$, $y=y_0$ and $s=\gamma_0$ in Lemma \ref{Lem: prox_bound}, we obtain
$$F(x_0) 
\le F^* + \frac{1}{s}\langle y_0-x_0, y_0-x^* \rangle - \frac{1}{2s}\| y_0-x_0 \|^2 
= F^* - \frac{1}{2s}\| x_0 - x^* \|^2 + \frac{1}{2s}\| y_0-x^* \|^2,$$
so that 
$$ F(x_0) - F^* + \frac{1}{2s}\| x_0-x^* \|^2 \le \frac{1}{2s}\| y_0-x^* \|^2.  $$
As a result, we obtain
$$ F(x_k) - F^* \le \frac{\| y_0-x^* \|^2}{2st_k^2}. $$
Since $t_k\to\infty$ as $k\to\infty$, the weak lower-semicontinuity of $F$ gives the minimizing property.\wbx

\begin{remark}\label{Rem: sequence_sum_limit}
From \eqref{E: E_k_diff_bound_temp}, we obtain
$$ st_k\eta_{k+1} + st_k\psi_k + \frac{1}{2}\left( 1 - \frac{1}{t_k} \right) \left\| t_k(y_{k+1}-y_k) - t_{k+1}(x_{k+1}-x_k) \right\|^2 \le E_k - E_{k+1}, $$
so that
$$ \sum_{k=0}^\infty t_k \eta_{k+1} < \infty,\quad \sum_{k=0}^\infty t_k\psi_k < \infty,\qbox{and} \sum_{k=0}^\infty \left\| t_k(y_{k+1}-y_k) - t_{k+1}(x_{k+1}-x_k) \right\|^2 < \infty. $$
This implies that
$$ \lim_{k\to\infty} t_k \eta_{k+1} 
= \lim_{k\to\infty} t_k\psi_k
= \lim_{k\to\infty} \left\| t_k(y_{k+1}-y_k) - t_{k+1}(x_{k+1}-x_k) \right\| = 0.$$
\end{remark}

\section{Boundedness and convergence of the iterates} \label{Sec: iterates}
We begin this section by establishing the following boundedness result:

\begin{proposition}\label{Prop: bounded_sequence}
The sequences $(x_k)_{k\ge 0}$, $(y_k)_{k\ge 0}$ and $(z_k)_{k\ge 0}$, generated by algorithm \eqref{Algo: APM}, are bounded. 
\end{proposition}

\begin{proof}
Recalling Remark \ref{Rem: R_k_bound} and writing $\xi_k = (t_k-1)(y_{k+1}-y_k) + x_k$, we obtain
$$ R_k = \frac{1}{2}\| \xi_k - x^* \|^2 + \frac{t_k-1}{2}\| y_{k+1}-y_k \|^2. $$
Since the sequence $(E_k)_{k\ge 0}$ is nonincreasing, we have
$$ \frac{1}{2}\| \xi_k - x^* \|^2 \le E_0, $$
which implies that $\| \xi_k \| \le \| x^* \| + \sqrt{2E_0}=: M$. Using $x_k = y_{k+1}+s\nabla f(x_k)$, we obtain
$$ \xi_k = t_k y_{k+1} - (t_k-1)y_k + s\nabla f(x_k), $$
so that
$$ y_{k+1} = \left( 1 - \frac{1}{t_k} \right) y_k + \frac{1}{t_k}\xi_k + \frac{s}{t_k}\nabla f(x_k). $$ 
This gives
$$ \| y_{k+1} \| \le \left( 1 - \frac{1}{t_k} \right)\| y_k \| + \frac{M}{t_k} + \frac{1}{t_k}\| s\nabla f(x_k) \|. $$
Since $f$ is convex and $L$-smooth, we have
\begin{equation}\label{E: gradient_f}
\frac{1}{2L}\| \nabla f(x_k) - \nabla f(x^*) \|^2 \le \psi_k \le \frac{E_k}{s(t_k-1)} \le \frac{E_0}{s(t_k-1)},\ \forall\, k\ge 1,
\end{equation}
which implies that $\nabla f(x_k)\to\nabla f(x^*)$ as $k\to\infty$, and that $\| \nabla f(x_k) \| \le B$ for all $k\ge 0$ and some $B>0$. As a result,
$$ \| y_{k+1} \| 
\le \max\left\{ \| y_k \|, M, sB \right\}
\le \max\left\{ \| y_0 \|, M, sB \right\}, $$
and so $(y_k)_{k\ge 0}$ is bounded. Since $x_k = y_{k+1} + s\nabla f(x_k)$, we have
$$ \| x_k \| \le \| y_{k+1} \| + s\| \nabla f(x_k) \|, $$
so that $(x_k)_{k\ge 0}$ is also bounded. Notice that
$$ \| -(x_{k+1}-y_{k+1}) + \lambda_{k+1} (y_{k+1}-y_k) \| \le \| x_{k+1} \| + \| y_{k+1} \| + \| y_{k+1}-y_k \| \le D<\infty,$$
for some $D>0$. Using \eqref{E: subgradient_iterate}, we arrive at
\begin{align*}
\left\| \frac{z_{k+1}-x_{k+1}}{\gamma_{k+1}} \right\| 
&\le \frac{ D }{(1 + \lambda_{k+1})s} + \left( \frac{\lambda_{k+1}}{1 + \lambda_{k+1}} \right) \left\| \frac{z_k-x_k}{\gamma_k} \right\| \\
&\le \max\left\{ \left\| \frac{z_k-x_k}{\gamma_k} \right\|, \frac{D}{s}  \right\}.
\end{align*}
This implies that $\left\| \frac{z_k-x_k}{\gamma_k} \right\|$ is also bounded. Since $(x_k)$ is bounded and $\gamma_k\ge s$ for all $k$, $(z_k)$ must be bounded as well.
\end{proof}

\begin{remark} \label{Rem: yk_regular}
Using Remark \ref{Rem: R_k_bound}, and the fact that the sequence $(E_k)_{k\ge 0}$ is nonincreasing, we deduce that
$$ \frac{t_k-1}{2}\| y_{k+1}-y_k \|^2 \le R_k \le E_k \le E_0.$$
This shows that $\lim_{k\to\infty}\|y_{k+1}-y_k\|=0$.
\end{remark}

\begin{remark} \label{Rem: grad_converges}
In view of \eqref{E: gradient_f}, as $t\to\infty$, $\nabla f(x_k)$ converges (strongly) to $\nabla f(x^*)$. This does not depend on the $x^*\in\mathcal S$ chosen, as seen in Remark \ref{Remark:nablaf_constant}. 
\end{remark}

\subsection{Weak convergence of the iterates}
The following results will be used to prove the convergence of the sequences generated by \eqref{Algo: APM}:

\begin{lemma}[Lemma A.4 in \cite{Radu_2025}]\label{Lem: convergence_limit}
Let $(u_k)_{k\ge 0}$ be a real sequence and $(\zeta_k)_{k\ge 0}$ be positive such that $\sum_{k=0}^{\infty}\frac{1}{\zeta_k}=\infty$. If
$ \lim_{k\to\infty}[ u_{k+1} + \zeta_k( u_{k+1} - u_k ) ] = b\in\R, $
then,
$ \lim_{k\to\infty}u_k = b. $ 
\end{lemma}

\begin{lemma} \label{Lem: convergence_vector}
Let $(W_k)_{k\ge 0}$ be a sequence in $H$, and let $(\zeta_k)_{k\ge 0}$ be a positive sequence such that $\sum_{k=0}^{\infty}\frac{1}{\zeta_k}=\infty$. If
$\lim_{k\to\infty}\big[W_{k+1} + \zeta_k( W_{k+1} - W_k )\big]=W\in H$, then  $\lim_{k\to\infty}W_k=W$.
\end{lemma}

\begin{proof}
Write $u_k=\|W_k-W\|$. By the (reverse) triangle inequality, we know that
\begin{align*}
    \big|u_{k+1} + \zeta_k( u_{k+1} - u_k )\big| & = \Big|\big\|(1+\zeta_k)(W_{k+1}-W)\|- \big\|\zeta_k (W_k-W)\big\| \Big| \\
    & \le \Big\|W_{k+1} + \zeta_k( W_{k+1} - W_k )-W\Big\|.
\end{align*}
Since the right-hand side goes to zero, Lemma \ref{Lem: convergence_limit} shows that $\lim_{k\to\infty}\|W_k-W\|= 0$.
\end{proof}

Now we are in a position to prove the weak convergence of the iterates.

\begin{theorem}
Consider algorithm \eqref{Algo: APM}, defined as in Theorem \ref{Thm: rate_cvx}. Then, $x_k$ converges weakly, as $k\to\infty$, to a minimizer $x^*$ of $F$. Moreover, $y_k$ and $z_k$ converge weakly, as $k\to\infty$, to $y^*:=x^*-s\nabla f(x^*)$ and $z^*:=x^*-2s\nabla f(x^*)$, respectively.  
\end{theorem}

\begin{proof}
By Theorem \ref{Thm: rate_cvx}, every weak subsequential limit point of $x_k$, as $k\to\infty$, is a minimizer of $F$. Recalling that from Proposition \ref{Prop: bounded_sequence}, $(x_k)_{k\ge 0}$ is bounded, we only need to prove that there cannot be two such limit points. To this end, suppose that $x_{m_k}\rightharpoonup \xi^*$ and $x_{n_k}\rightharpoonup\widetilde{\xi}^*$, as $k\to\infty$. Then $\xi^*,\widetilde{\xi}^*\in\argmin(F)$. Our objective is to show $\xi^* = \widetilde{\xi}^*$. Define
$$ U_k = E_k(\xi^*) - E_k(\widetilde{\xi}^*). $$
Since $(E_k)_{k\ge 0}$ is nonincreasing and bounded from below, $\lim_{k\to\infty} U_k$ exists. From Remark \ref{Rem: sequence_sum_limit}, we know that    
$$ \lim_{k\to\infty} t_k \eta_{k+1} =  \lim_{k\to\infty} t_k\psi_k = \lim_{k\to\infty} \left\| t_k(y_{k+1}-y_k) - t_{k+1}(x_{k+1}-x_k) \right\| = 0. $$
From the definition of $E_k$ in \eqref{E: E_k}, we obtain
\begin{align*}
 \lim_{k\to\infty} U_k 
&= \lim_{k\to\infty}\left[ -(t_k-1)\langle y_{k+1}-y_k, \xi^* - \widetilde{\xi}^* \rangle + \frac{1}{2}\left\| x_k - \xi^* \right\|^2 - \frac{1}{2}\left\| x_k - \widetilde{\xi}^* \right\|^2  \right] \\
&= \lim_{k\to\infty}\left[ -\left(t_{k+1}-\frac{t_{k+1}}{t_k}\right)\langle x_{k+1}-x_k, \xi^* - \widetilde{\xi}^* \rangle + \frac{1}{2}\left\| x_k - \xi^* \right\|^2 - \frac{1}{2}\left\| x_k - \widetilde{\xi}^* \right\|^2 \right].
\end{align*}
Defining
$$ u_k = \frac{1}{2}\left\| x_k - \xi^* \right\|^2 - \frac{1}{2}\left\| x_k - \widetilde{\xi}^* \right\|^2 = -\left\langle x_k, \xi^* - \widetilde{\xi}^* \right\rangle + \frac{1}{2}\| \xi^* \|^2 - \frac{1}{2}\left\| \widetilde{\xi}^* \right\|^2, $$
we get
\begin{align*}
\lim_{k\to\infty} U_k 
&= \lim_{k\to\infty}\left[ \left(t_{k+1}-\frac{t_{k+1}}{t_k}\right)(u_{k+1}-u_k) + u_k \right] \\
&= \lim_{k\to\infty}\left[ u_{k+1} + \left(t_{k+1}-\frac{t_{k+1}}{t_k}-1\right)(u_{k+1}-u_k) \right].
\end{align*}
Invoking Lemma \ref{Lem: convergence_limit}, we deduce that $\lim_{k\to\infty}u_k$ exists, and equals $\lim_{k\to\infty}U_k$. Replacing $k$ by $m_k$, and then by $n_k$ in the definition of $u_k$, and taking limits, we obtain
$$ -\frac{1}{2}\left\| \xi^* - \widetilde{\xi}^* \right\|^2 = \lim_{k\to\infty} u_k = \frac{1}{2}\left\| \xi^* - \widetilde{\xi}^* \right\|^2, $$
which implies that $\xi^* = \widetilde{\xi}^*$. This means that $x_k$ converges weakly to a minimizer of $F$, which we denote by $x^*$. Since $y_{k+1}=x_k-s\nabla f(x_k)$, Remark \ref{Rem: grad_converges} shows that $y_k$ converges weakly to $y^*=x^*-s\nabla f(x^*)$. For $z_k$, we rewrite the second line of \eqref{Algo: APM} as
$$ (1+\lambda_{k+1})(z_{k+1}-y_{k+1}) = \lambda_{k+1}(1+\lambda_{k+1})(y_{k+1}-y_k) + \lambda_{k+1}(z_k-x_k),  $$
which is equivalent to
\begin{align*}
(1+\lambda_{k+1})(z_{k+1}-y_{k+1})-\lambda_{k+1}(z_k-y_k) 
& = \lambda_{k+1}^2(y_{k+1}-y_k)+\lambda_{k+1}(y_{k+1}-x_k) \\
& = \lambda_{k+1}^2(y_{k+1}-y_k)-\lambda_{k+1}s\nabla f(x_k),
\end{align*}
where we have used $y_{k+1}-x_k=-s\nabla f(x_k)$. As a consequence of Remarks \ref{Rem: yk_regular} and \ref{Rem: grad_converges}, the right-hand side converges to $W:=-s\nabla f(x^*)$ as $k\to\infty$. Using Lemma \ref{Lem: convergence_vector} with $W_k:=z_k-y_k$, we deduce that $z_k$ converges weakly, as $k\to\infty$, to $z^*=y^*+W=x^*-2s\nabla f(x^*)$.
\end{proof}

\section{Convergence rate in the strongly convex case}\label{Sec: strong convexity}
In this section, we establish an accelerated linear convergence rate for \eqref{Algo: APM} with
$$ \lambda_{k+1} = \frac{1-\theta}{1+\theta},\quad \theta=\sqrt{\mu s},\quad s\in\left( 0,\frac{1}{L} \right], $$
when $f$ is $\mu$-strongly convex for some $\mu>0$. Algorithm \eqref{Algo: APM} has not been studied in this context before. Our analysis centers around the following energy sequence $(E_k)_{k\ge 0}$:
\begin{equation}\label{E: E_k_SC}
E_k = F(x_k) - F^* + \frac{\theta}{1+\theta}\eta_k + \frac{\theta}{1+\theta}\psi_k + R_k,
\end{equation}
where
$$ R_k = \frac{\| y_{k+1} - y_k \|^2}{2s(1+\theta)} + \frac{\theta}{s(1+\theta)}\langle y_{k+1}-y_k, x_k-x^* \rangle + \frac{\theta^2 \| x_k - x^* \|^2 }{2s(1+\theta)}. $$

\begin{remark}
Notice that
$$R_k = \frac{1}{2s(1+\theta)}\| (y_{k+1}-y_k) + \theta(x_k-x^*) \|^2 \ge 0.$$
Hence, we have $E_k \ge 0$ and $F(x_k)-F^*\le E_k$.
\end{remark}

The following result shall be useful in the forthcoming convergence proof.

\begin{lemma}\label{Lem: quadratic_positive}
Let $a,c\ge 0$ and $b\in\mathbb{R}$ satisfy $b^2\le 4ac$. Then, $a\| \xi \|^2 + b\langle \xi,\zeta \rangle + c\| \zeta \|^2 \ge 0$ for every $\xi,\zeta\in H$.
\end{lemma}

\begin{proof}
Let $W(\xi,\zeta) = a\| \xi \|^2 + b\langle \xi,\zeta \rangle + c\| \zeta \|^2$. If $a=0$, then $b=0$, and so $W(\xi,\zeta) = c\| \zeta \|^2 \ge 0$. On the other hand, if $a> 0$, we have $W(\xi,\zeta) = a \left\| \xi + \frac{b\zeta}{2a} \right\|^2 + \frac{4ac-b^2}{4a}\| \zeta \|^2 \ge 0$. 
\end{proof}

The linear convergence result is as follows:

\begin{theorem}
Let $F=f+g$, where $f:H\to\mathbb{R}$ is $\mu$-strongly convex and $L$-smooth, and $g:H\to\mathbb{R}\cup\{\infty\}$ is convex, proper and lower-semicontinuous. Let $(x_k)_{k\ge 0}$, $(y_k)_{k\ge 0}$ and $(z_k)_{k\ge 0}$ be generated by algorithm \eqref{Algo: APM}, where
$$ s\in\left( 0,\frac{1}{L} \right],\quad \lambda_{k+1} = \frac{1-\theta}{1+\theta},\quad \theta = \sqrt{\mu s}, $$ 
and initial conditions
$$ \gamma_0 = s,\quad z_0\in H,\quad x_0=\prox_{\gamma_0g}(z_0),\quad y_0=x_0-s\nabla f(x_0).$$
For every $k\ge 0$, we have
$$ F(x_k) - F^* \le (1-\theta)^k\left[ F(x_0) - F^* + \frac{\theta}{1+\theta}\eta_0 + \frac{\theta}{2s}\| x_0 - x^* \|^2 \right]. $$
\end{theorem}

\begin{proof}
Using \eqref{E: E_k_SC}, we have
\begin{equation}\label{E: E_k_diff_SC}
\begin{aligned}
&\quad E_{k+1} - (1-\theta)E_k \\
&= (1-\theta)\left(F(x_{k+1}) - F(x_k)\right) + \theta\left( F(x_{k+1}) - F^* \right)
   + \frac{\theta}{1+\theta}\eta_{k+1}  
   - \frac{\theta(1-\theta)}{1+\theta}\eta_k\\
&\quad + \frac{\theta}{1+\theta}\psi_{k+1} 
  - \frac{\theta(1-\theta)}{1+\theta}\psi_k 
  + \frac{1}{2s(1+\theta)}\| y_{k+2} - y_{k+1} \|^2 - \frac{1-\theta}{2s(1+\theta)}\| y_{k+1} - y_k \|^2 \\
&\quad  + \frac{\theta}{s(1+\theta)} \langle y_{k+2}-y_{k+1}, x_{k+1}-x^* \rangle - \frac{\theta(1-\theta)}{s(1+\theta)}\langle y_{k+1} - y_k, x_k-x^* \rangle \\
&\quad + \frac{\theta^2}{2s(1+\theta)}\| x_{k+1} - x^* \|^2 - \frac{\theta^2(1-\theta)}{2s(1+\theta)}\| x_k - x^* \|^2.
\end{aligned}
\end{equation}
Setting $\lambda_{k+1} = \frac{1-\theta}{1+\theta}$ in Proposition \ref{Prop: F_estimate} and keeping in mind that $\theta^2=\mu s$, we obtain
\begin{equation}\label{E: func_gap_SC}
\begin{aligned}
&\quad (1-\theta)\left(F(x_{k+1}) - F(x_k) \right) + \frac{1}{2s(1+\theta)}\| y_{k+2} - y_{k+1} \|^2 - \frac{1-\theta}{2s(1+\theta)}\| y_{k+1} - y_k \|^2 \\
&\le - \frac{1-\theta}{2s(1+\theta)}\| y_{k+1} - y_k \|^2 - \frac{1-\theta}{2s}\| x_{k+1}-x_k \|^2 
+ \frac{(1-\theta)^2}{s(1+\theta)}\langle y_{k+1}-y_k, x_{k+1}-x_k \rangle, 
\end{aligned}
\end{equation} 
and
\begin{equation}\label{E: func_value_SC}
\begin{aligned}
&\quad \theta\left( F(x_{k+1}) - F^* \right) + \frac{2\theta}{1+\theta}\eta_{k+1} - \frac{\theta(1-\theta)}{1+\theta}\eta_k + \theta\psi_{k+1}\\
&\le -\frac{\theta}{s}\langle y_{k+2}-y_{k+1}, x_{k+1}-x^* \rangle + \frac{\theta(1-\theta)}{s(1+\theta)}\langle y_{k+1}-y_k, x_{k+1}-x^* \rangle \\
&= -\frac{\theta}{s(1+\theta)}\langle y_{k+2}-y_{k+1}, x_{k+1}-x^* \rangle
   + \frac{\theta(1-\theta)}{s(1+\theta)}\langle y_{k+1}-y_k, x_k-x^* \rangle \\
&\quad -\frac{\theta^2}{s(1+\theta)}\langle y_{k+2}-y_{k+1}, x_{k+1}-x^* \rangle
   + \frac{\theta(1-\theta)}{s(1+\theta)}\langle y_{k+1}-y_k, x_{k+1}-x_k \rangle. 
\end{aligned}
\end{equation}
For the second to last term in \eqref{E: func_value_SC}, we substitute \eqref{E: key_inequality}, to obtain
\begin{align*}
&\quad \theta\left( F(x_{k+1}) - F^* \right) + \frac{2\theta}{1+\theta}\eta_{k+1} - \frac{\theta(1-\theta)}{1+\theta}\eta_k 
   + \frac{\theta}{1+\theta}\psi_{k+1}
   - \frac{\theta(1-\theta)}{1+\theta}\psi_k \\
&\quad +\frac{\theta}{s(1+\theta)}\langle y_{k+2}-y_{k+1}, x_{k+1}-x^* \rangle
   - \frac{\theta(1-\theta)}{s(1+\theta)}\langle y_{k+1}-y_k, x_k-x^* \rangle \\
&\quad + \frac{\theta^2}{2s(1+\theta)}\| x_{k+1} - x^* \|^2
   - \frac{\theta^2(1-\theta)}{2s(1+\theta)}\| x_k - x^* \|^2 \\
&\le \frac{\theta(1-\theta)}{s(1+\theta)}\langle y_{k+1}-y_k, x_{k+1}-x_k \rangle
   + \frac{\theta^3}{2s(1+\theta)}\| x_k - x^* \|^2 
   - \frac{\theta}{1+\theta}\psi_k. 
\end{align*}
Using this inequality and \eqref{E: func_gap_SC} in \eqref{E: E_k_diff_SC}, we obtain
\begin{align*}
&\quad E_{k+1}-(1-\theta)E_k \\
&\le - \frac{\theta}{1+\theta}\eta_{k+1}
     + \frac{\theta^3}{2s(1+\theta)}\| x_k - x^* \|^2
     - \frac{\theta}{1+\theta}\psi_k 
     - \frac{1-\theta}{2s(1+\theta)}\| y_{k+1} - y_k \|^2\\
&\quad  
    - \frac{1-\theta}{2s}\| x_{k+1}-x_k \|^2   
    + \frac{1-\theta}{s(1+\theta)}\langle y_{k+1}-y_k, x_{k+1}-x_k \rangle \\
&\le - \frac{\theta}{1+\theta}\eta_{k+1}
     + \frac{\theta^3}{2s(1+\theta)}\| x_k - x^* \|^2
     - \frac{\theta}{1+\theta}\psi_k\\
&\le 0,   
\end{align*}
where the second inequality is due to $\theta\le 1$ and Lemma \ref{Lem: quadratic_positive}, and the last inequality is due to strong convexity of $f$, which implies that $\psi_k \ge \frac{\mu}{2}\| x_k - x^* \|^2 = \frac{\theta^2}{2s}\| x_k - x^* \|^2$. Hence, we obtain $E_{k+1} \le (1-\theta)E_k$ for all $k\ge 0$. Iterating, we obtain
$$ F(x_k) - F^* \le E_k \le (1-\theta)^kE_0. $$ 
With $y_0 = x_0-s\nabla f(x_0)=y_1$ in mind, we have
\begin{align*}
E_0
&= F(x_0) - F^* + \frac{\theta}{1+\theta}\eta_0 + \frac{\theta}{1+\theta}\psi_0 + \frac{\theta^2 }{2s(1+\theta)}\| x_0 - x^* \|^2 \\
&\le F(x_0) - F^* + \frac{\theta}{1+\theta}\eta_0 + \frac{\theta}{2s(1+\theta)}\| x_0 - x^* \|^2 
 + \frac{\theta^2 }{2s(1+\theta)}\| x_0 - x^* \|^2 \\
&= F(x_0) - F^* + \frac{\theta}{1+\theta}\eta_0 + \frac{\theta}{2s}\| x_0 - x^* \|^2,    
\end{align*}
where the inequality is due to $L$-smoothness of $f$ and $s\le \frac{1}{L}$. Hence, we obtain the desired result.
\end{proof}

\begin{remark}
Setting $s=\frac{1}{L}$, the convergence rate $\mathcal{O}\left( (1-\sqrt{\mu/L})^{k} \right)$ matches the one for \eqref{Algo: FISTA-SC}. 
\end{remark}

\section{Conclusions}\label{Sec: conclusions}
We provided a convergence proof for the $\mathcal{O}\left( \frac{1}{k^2} \right)$ rate of the accelerated backward forward algorithm \eqref{Algo: APM} when $f$ is convex, and established the weak convergence of the iterates. When $f$ is strongly convex, we also developed a variant of \eqref{Algo: APM}, for which we proved an accelerated linear convergence rate for the function values.

\begin{acknowledgements}
This work was partially funded by the China Scholarship Council~202208520010, and also benefited from the support of the FMJH Program Gaspard Monge for optimization and operations research and their interactions with data science.
\end{acknowledgements}

\bibliography{myrefs}

\end{document}